\documentclass[12pt]{amsart}
\usepackage{amscd}
\usepackage{amssymb,amsmath}
\usepackage{hyperref}
\usepackage{mathrsfs}
\usepackage{graphicx}
\usepackage{enumitem}
\usepackage{romannum}
\usepackage{mathtools}
\usepackage[utf8]{inputenc}
\usepackage{listings}

\usepackage{lipsum,eso-pic,xcolor}
\usepackage{lineno}
\usepackage{tikz}
\usetikzlibrary{arrows,decorations.pathmorphing,backgrounds,positioning,fit}
\usetikzlibrary{positioning}
\hypersetup{
	colorlinks=true,
	linkcolor=blue,}

\usetikzlibrary{trees}
\usetikzlibrary{arrows}

\def\v{\operatorname{v}}

\def\supp{\operatorname{supp}}

\def\deg{\operatorname{deg}}
\def\Min{\operatorname{Min}}
\def\Ass{\operatorname{Ass}}
\def\max{\operatorname{max}}

\def\Emb{\operatorname{Emb}}

\newtheorem{lemma}{Lemma}[section]
\newtheorem{corollary}[lemma]{Corollary}
\newtheorem{theorem}[lemma]{Theorem}

\newtheorem{definition}[lemma]{Definition}
\newtheorem{remark}[lemma]{Remark}

\advance\headheight1.15pt

\begin{document}
	
	\pagenumbering{arabic}
	
	\title[v-number]{V-numbers of symbolic-ordinary discrepancy modules} 
    
	\author[Manohar Kumar]{Manohar Kumar$^*$}
	\address{Department of Mathematics, Indian Institute of Technology Madras, Chennai, INDIA - 600036.}
	\email{manhar349@gmail.com}

\author[Nguyen Cong Minh]{Nguyen Cong Minh}
	\address{Faculty of Mathematics and Informatics, Hanoi University of Science and Technology, Vietnam}
	\email{minh.nguyencong@hust.edu.vn}

\author{Thanh Vu}
\address{Institute of Mathematics, VAST, 18 Hoang Quoc Viet, Hanoi, Vietnam}
\email{vuqthanh@gmail.com}

	\thanks{AMS Classification 2020: 13D02, 05E40, 05E99, 13D45}
	\thanks{$^*$ Corresponding author}

	\maketitle
	
	\begin{abstract}
		In this paper, we study the ${\rm v}$-number of the symbolic-ordinary discrepancy modules $M_t(I) = I^{(t)} / I^t$ over standard graded Noetherian rings. We show that this coincides with the ${\rm v}$-number of $I^t$ at embedded primes. Moreover, under mild assumptions on $I$, we show that the ${\rm v}$-number of $N_t(I) = I^{(t)} / I^{(t+1)}$ is the same as the ${\rm v}$-number of $I^{(t+1)}$. Furthermore, we provide combinatorial formulas for these ${\rm v}$-numbers when $t = 2, 3$ and $I$ is the edge ideal of a graph. Consequently, we classify all graphs for which ${\rm v}(I^2) = {\rm v}(I^{(2)})$.
	\end{abstract}
\section{Introduction}
Let $R$ be a standard graded Noetherian ring with $R_0$ a field, and let $M$ be a finitely generated graded $R$-module. For an associated prime $P$ of $M$, the \textit{local $\v$-number} of $M$ at $P$ is defined by
\[
\v_P(M) = \min\{\deg m \mid m \in M,\ \operatorname{ann}_R(m) = P\},
\]
and the \textit{$\v$-number} of $M$ is
\[
\v(M) = \min\{\v_P(M) \mid P \in \operatorname{Ass}(M)\}.
\]
The notion of $\v$-number was originally introduced for ideals by Cooper, Seceleanu, Tohaneanu, Vaz Pinto, and Villarreal \cite{cstvv20} in connection with the minimum distance function of projective Reed--Muller-type codes. Since then, the invariant has been studied from both algebraic and geometric perspectives and has found applications in several areas of commutative algebra. We refer the reader to \cite{chjv26,c24,fs25,fs26,kns25} for further developments in this direction. In particular, recent work has established asymptotic linearity of $\v$-functions associated to powers of ideals and to graded (co)homology modules; see, for example, \cite{fg25,gp25}. We will use the module-theoretic formulation above throughout the paper.

Let $I$ be a graded ideal of $R$. For $t \ge 1$, consider the ordinary and symbolic powers $I^t$ and $I^{(t)}$. The comparison between these two filtrations is a central theme in the study of powers of ideals. To capture these underlying structural differences, one considers two naturally associated graded modules
\[
M_t = I^{(t)} / I^t \qquad\text{and}\qquad N_t = I^{(t)} / {I^{(t+1)}}.
\]
The first, which we call the \textit{discrepancy module}, measures precisely the difference between the ordinary and symbolic powers (\cite{hm26}). In particular, its associated primes reflect the embedded primes that arise in the ordinary power $I^t$. Consequently, the $\v$-number of $M_t$ provides a numerical measure of this discrepancy and allows one to study the difference between the $\v$-numbers of $I^t$ and $I^{(t)}$ through a module supported on the embedded-prime locus. The second module, $N_t$, is the $t$-th graded piece of the associated graded object of the symbolic-power filtration and therefore records the successive layers of the symbolic powers. Its $\v$-number is consequently related to the local behavior of the symbolic powers at their minimal primes.

The study of these modules is motivated in part by the recent work of Chau, Hà, Jayanthan, and Vu \cite{chjv26}, which initiated the investigation of the relationship between the $\v$-numbers of ordinary and symbolic powers of squarefree monomial ideals. Their results suggest that discrepancy modules provide a natural framework for understanding this relationship. In this paper, we develop this viewpoint systematically by studying the $\v$-numbers of $M_t$ and $N_t$.

Our main results can be summarized as follows:
\begin{enumerate}
    \item We determine the associated primes of the discrepancy modules $M_t$ (Theorem~\ref{thm_ass_M}) and the successive symbolic-power quotients $N_t$ (Theorem~\ref{thm_ass_N}), under mild hypotheses on $I$.
    \item We establish reductions that express the local $\v$-numbers of $M_t$ in terms of the local $\v$-numbers of $I^t$ at embedded associated primes (Theorem~\ref{thm_v_M}). Similarly, under suitable hypotheses, the local $\v$-numbers of $N_t$ are determined by those of the symbolic powers $I^{(t+1)}$ at minimal primes (Theorem~\ref{thm_v_N}).
    \item When $I = I(G)$ is the edge ideal of a graph $G$, we obtain explicit combinatorial formulas for the $\v$-numbers of the discrepancy modules $M_2(I(G))$ (Theorem~\ref{thm_v_2}) and $M_3(I(G))$ (Theorem~\ref{thm_v_3}).
    \item As applications of these formulas, we obtain graph-theoretic characterizations of when the $\v$-numbers of the second ordinary and symbolic powers coincide (Theorem~\ref{thm_equality_v_2}).
\end{enumerate}

The paper is organized as follows. In Section~\ref{sec_assoc}, we study the associated primes of the modules $M_t$ and $N_t$. In Section~\ref{sec_v}, we establish the reductions for their local and global $\v$-numbers. In Section~\ref{sec_edge}, we specialize to edge ideals of graphs and derive combinatorial formulas for $M_2(I(G))$ and $M_3(I(G))$, together with their applications.

\section{Associated primes of discrepancy modules}\label{sec_assoc}
In this section, we prove general results regarding the associated primes of $M_t$ and $N_t$. 

For an ideal $I$, the $t$-th symbolic power of $I$ is defined as
\[
I^{(t)} = \left( \bigcap_{P \in \operatorname{Ass}(R/I)} I^t R_P \right) \cap R.
\]
The discrepancy module $M_t = I^{(t)} / I^t$ arises naturally when comparing the invariants of $I^t$ with those of $I^{(t)}$. Similarly, the module $N_t = I^{(t)} / I^{(t+1)}$ corresponds to the $t$-th graded component of the associated graded ring of the symbolic Rees algebra of $I$. We denote by $\operatorname{Emb}_t(I) = \operatorname{Ass}(R/I^t) \setminus \operatorname{Min}(I)$ the set of embedded primes of $I^t$.

We first establish a general result regarding the associated primes of the discrepancy module $M_t$.

\begin{theorem}\label{thm_ass_M} 
Let $R$ be a Noetherian ring and let $I$ be a proper ideal of $R$. Assume that $I$ has no embedded primes. Then:
\begin{enumerate}
    \item $I^t = I^{(t)}$ if and only if $I^t$ has no embedded primes.
    \item $\operatorname{Ass}(M_t) = \operatorname{Emb}_t(I)$.
\end{enumerate}    
\end{theorem}
\begin{proof}
(1) We first note that $\Ass(R/I^{(t)}) = \Ass(R/I) = \Min(I)$. Indeed, for each $P \in \Ass(R/I) = \Min(I)$, the radical $\sqrt{I^t R_P} = PR_P$ is the maximal ideal of $R_P$. Hence, $I^t R_P$ is $PR_P$-primary, which implies that $Q_P = I^t R_P \cap R$ is $P$-primary. Since $R/I^{(t)} \hookrightarrow \bigoplus_{P \in \Min(I)} R/Q_P$, we deduce that $\Ass(R/I^{(t)}) \subseteq \Min(I)$. Moreover, because $\sqrt{I^{(t)}} = \sqrt{I}$, it follows that $\Ass(R/I^{(t)}) = \Ass(R/I) = \Min(I)$.

For each $P \in \Ass(R/I) = \Min(I)$, let $Q_P = I^t R_P \cap R$. Then $Q_P$ is the $P$-primary component of $I^t$. Therefore, $I^t = I^{(t)}$ if and only if $I^t$ has no embedded primes.

(2) From the short exact sequence 
$$0 \to M_t \to R/I^t \to R/I^{(t)} \to 0,$$
we obtain the inclusions $\Ass(M_t) \subseteq \Ass(R/I^t) \subseteq \Ass(M_t) \cup \Ass(R/I^{(t)})$. For each $P \in \Min(I)$, we have $I^{(t)} R_P = I^t R_P$, which implies that $(M_t)_P = 0$ and thus $\Ass(M_t) \cap \Min(I) = \emptyset$. Since $\Ass(R/I^{(t)}) = \Min(I)$, it follows that $\Ass(M_t) = \Ass(R/I^t) \setminus \Min(I) = \Emb_t(I)$, as desired.
\end{proof}

We note that the assumption that $I$ has no embedded primes is not strictly necessary if symbolic powers are defined via intersection over minimal primes. For convenience, we assume throughout that $I$ has no embedded primes so that the two standard definitions of symbolic powers coincide.

\begin{corollary}
Let $R$ be a Noetherian ring and $I$ a proper ideal of $R$. Assume that $I$ has no embedded primes. Then 
$$\dim(M_t) = \max \{ \dim(R/P) \mid P \in \operatorname{Emb}_t(I) \}.$$
\end{corollary}
\begin{proof}
The conclusion follows immediately from Theorem~\ref{thm_ass_M}.
\end{proof}

We will apply this to give a combinatorial description of $\dim(M_t(I))$ when $I$ is the edge ideal of a simple graph. We first establish a preliminary lemma.

\begin{lemma}\label{lem_chain_localization}
Let $R$ be a Noetherian ring and let $P \subseteq Q$ be prime ideals of $R$. Then for any ideal $J$ of $R$, we have 
$$(J R_P \cap R) R_Q = J R_P \cap R_Q.$$
\end{lemma}

\begin{proof}
Let $\varphi \colon R \to R_P$ and $\psi \colon R \to R_Q$ be the canonical localization maps, and let $\theta \colon R_Q \to R_P$ be defined by $\theta(x/s) = x/s$. This map is well-defined since $s \notin Q$ implies $s \notin P$. By definition, we have $\varphi = \theta \circ \psi$. 

Set $K = J R_P \cap R = \varphi^{-1}(J R_P)$. 

Let $x \in K$. Then $\theta(\psi(x)) = \varphi(x) \in J R_P$, which implies $\psi(x) \in \theta^{-1}(J R_P)$. Since $K R_Q = \psi(K) R_Q$, it follows that $K R_Q \subseteq \theta^{-1}(J R_P) = J R_P \cap R_Q$.

Conversely, let $x/s \in R_Q$ such that $\theta(x/s) \in J R_P$. Then $\varphi(x) = \theta(x/s) \varphi(s) \in J R_P$, which implies $x \in \varphi^{-1}(J R_P) = K$. Therefore, in $R_Q$, we have $x/s = \psi(x)/s \in K R_Q$. This completes the proof.
\end{proof}

\begin{lemma}\label{lem_M_localization} 
Let $R$ be a Noetherian ring and $I$ a non-zero proper ideal of $R$. Assume that $I$ has no embedded primes. For any prime ideal $Q$ of $R$, we have 
$$(M_t)_Q = (I R_Q)^{(t)} / (I R_Q)^t.$$    
\end{lemma}

\begin{proof} 
For any ring homomorphism $\varphi \colon R \to S$ and ideals $I, J$ of $R$, we have $\varphi(I J) = \varphi(I) \varphi(J)$. Consequently, $I^t R_Q = (I R_Q)^t$. 

Since localization commutes with finite intersections, we have 
$$I^{(t)} R_Q = \left( \bigcap_{P \in \Min(I)} (I^t R_P \cap R) \right) R_Q = \bigcap_{P \in \Min(I)} \left( (I^t R_P \cap R) R_Q \right).$$
Note that $(I^t R_P \cap R)$ is $P$-primary. If $P \not\subseteq Q$, then there exists an element $s \in P \setminus Q$. A power of $s$ lies in $(I^t R_P \cap R)$, which becomes a unit in $R_Q$; hence, $(I^t R_P \cap R) R_Q = R_Q$. 

If $P \subseteq Q$, then by Lemma~\ref{lem_chain_localization}, we have $(I^t R_P \cap R) R_Q = I^t R_P \cap R_Q = (I R_Q)^t (R_Q)_{P R_Q} \cap R_Q$. Taking the intersection over all minimal primes of $I R_Q$, it follows that $I^{(t)} R_Q = (I R_Q)^{(t)}$. The conclusion then follows by localizing the exact sequence defining $M_t$.
\end{proof}

We now turn to the associated primes of $N_t$. We first recall a preliminary lemma (see \cite[Proposition 5, page 263]{B}).

\begin{lemma}\label{lem_ass_localization} 
Let $R$ be a Noetherian ring, $M$ an $R$-module, and $S$ a multiplicatively closed subset of $R$. Then 
$$\Ass_{S^{-1} R}(S^{-1} M) = \{ S^{-1} P \mid P \in \Ass_R(M) \text{ and } S \cap P = \emptyset \}.$$
\end{lemma}

\begin{theorem}\label{thm_ass_N} 
Let $R$ be a Noetherian ring and $I$ a non-zero proper ideal of $R$. Assume that $I$ has no embedded primes and $\operatorname{ht}(I) \ge 1$. Then 
$$\Ass(N_t) = \Ass(R/I) = \Min(I).$$
\end{theorem}
\begin{proof}
From the short exact sequence 
$$0 \to N_t \to R/I^{(t+1)} \to R/I^{(t)} \to 0,$$
we deduce that $\Ass(N_t) \subseteq \Ass(R/I^{(t+1)}) = \Ass(R/I)$. Now, for any $P \in \Ass(R/I) = \Min(I)$, localizing at $P$ yields $(N_t)_P \cong I^t R_P / I^{t+1} R_P$. By Lemma~\ref{lem_ass_localization}, the set of associated primes of $(N_t)_P$ satisfies $\Ass_{R_P}((N_t)_P) \subseteq \{ P R_P \}$. 

Since $\operatorname{ht}(I) \ge 1$, we have $I R_P \neq 0$ and thus $I^t R_P \neq 0$. By Nakayama's Lemma, it follows that $I^t R_P / I^{t+1} R_P \neq 0$. Hence, $\Ass_{R_P}((N_t)_P) = \{ P R_P \}$, which implies that $P \in \Ass(N_t)$. The conclusion follows.
\end{proof}

We also have the following result on the quotient module $I^r / I^q$.

\begin{theorem}\label{theorem_filtered_I}
Let $R$ be a Noetherian ring and let $I$ be a proper ideal of $R$. Assume that one of the following conditions holds:
\begin{enumerate}
    \item $\operatorname{depth}(\operatorname{gr}_I(R)) \ge 1$;
    \item $R$ is a polynomial ring over a field and $I$ is a monomial ideal.
\end{enumerate} 
Then for all $1 \le r < q$, we have 
$$\operatorname{Ass}(I^r / I^q) = \operatorname{Ass}(R / I^q).$$    
\end{theorem}

\begin{proof}
Since $I^r / I^q \hookrightarrow R / I^q$, we have $\operatorname{Ass}(I^r / I^q) \subseteq \operatorname{Ass}(R / I^q)$. Now, let $P \in \operatorname{Ass}(R / I^q)$ and let $f \in R$ be an element such that $I^q : f = P$. It suffices to show that $f \in I^r$. 

We have $P \supseteq \sqrt{I^q} = \sqrt{I}$. Since $P$ is a prime ideal, $P \supseteq I$. Hence,
$$f \in I^q : P \subseteq I^q : I = I^{q-1} \subseteq I^r.$$
The equality $I^q : I = I^{q-1}$ follows from the assumption $\operatorname{depth}(\operatorname{gr}_I(R)) \ge 1$, which implies that $I^k$ is Ratliff--Rush closed for all $k \ge 1$ by \cite[Lemma 3.10]{rv10}.

Now assume that $R$ is a polynomial ring over a field and $I$ is a monomial ideal. Since $P$ is an associated prime of a monomial ideal, $P$ is generated by variables. In particular, $P$ contains a variable $x_j$, so
$$f \in I^q : P \subseteq I^q : x_j \subseteq I^{q-1} \subseteq I^r.$$
The conclusion follows.
\end{proof}

\begin{remark}
Let $R$ be a Noetherian ring and let $\mathbb{I} = \{I_n\}_{n \ge 0}$ be a multiplicative filtration, i.e., $I_0 = R$, $I_{n+1} \subseteq I_n$, and $I_m I_n \subseteq I_{m+n}$ for all $m, n \ge 0$. Let $G = \operatorname{gr}_{\mathbb{I}}(R) = \bigoplus_{n \ge 0} I_n / I_{n+1}$ and $G_+ = \bigoplus_{n \ge 1} I_n / I_{n+1}$. 

The filtration $\mathbb{I}$ is called \textit{Noetherian} if the Rees algebra $\mathcal{R}(\mathbb{I}) = \bigoplus_{n \ge 0} I_n t^n$ is a finitely generated $R$-algebra. Equivalently, there exists a positive integer $d$ such that $I_{n+d} = I_n I_d$ for all sufficiently large $n$. Let $d$ be the smallest such integer, and set $\tilde{I}_n = \bigcup_{k \ge 0} (I_{n + kd} :_R I_d^k)$. Then $\tilde{I}_n = I_n$ for all $n \ge 1$ if and only if $H^0_{G_+}(G) = 0$. In this case, we have 
\[
\operatorname{Ass}(I_r / I_q) = \operatorname{Ass}(R / I_q)
\]
for all $0 < r \le q - d$. In this general formulation, the condition $r \le q - d$ is necessary; we leave the verification as an exercise for the interested reader.
\end{remark}

\section{v-number of discrepancy modules}\label{sec_v}

In this section, we assume that $R$ is a graded Noetherian ring and $I$ a proper graded ideal of $R$.

\begin{theorem}\label{thm_v_M} 
Assume that $R$ is a standard graded Noetherian ring and $I$ is a non-zero proper graded ideal of $R$. Assume that $I$ has no embedded primes. Then for every prime $P \in \Emb_t(I)$, we have
$$\v_P(M_t) = \v_P(I^t).$$    
\end{theorem}
\begin{proof}
It suffices to show that if $f$ is a homogeneous element of $R$ such that $I^t : f = P$, then $f \in I^{(t)}$. Equivalently, we must show that $f \in I^t R_Q \cap R$ for every minimal prime $Q$ of $I$. Since $P$ is an embedded prime of $I^t$, we have $P \not\subseteq Q$. Choose an element $s \in P \setminus Q$. Since $s f \in I^t$ and $s$ becomes a unit in $R_Q$, we deduce that $f = (s f)/s \in I^t R_Q$. The conclusion follows.
\end{proof}

\begin{theorem}\label{thm_v_N} 
Assume that $R$ is a standard graded Noetherian ring and $I$ is a non-zero proper graded ideal of $R$. Assume that $I$ has no embedded primes and $\operatorname{ht}(I) \ge 1$. Let $P$ be a minimal prime of $I$. Assume that either of the following conditions holds: 
\begin{enumerate}
    \item $\operatorname{depth}(\operatorname{gr}_{I R_P}(R_P)) \ge 1$, where $\operatorname{gr}_{I R_P}(R_P)$ is the associated graded ring of $I R_P$;
    \item $R$ is a polynomial ring and $I$ is a monomial ideal.
\end{enumerate} 
Then 
$$\v_P(N_t) = \v_P(I^{(t+1)}).$$    
\end{theorem}

\begin{proof}
Let $f$ be a homogeneous element of $R$ such that $I^{(t+1)} : f = P$. For any minimal prime $Q$ of $I$ with $Q \neq P$, choose an element $s \in P \setminus Q$. Since $s f \in I^{(t+1)} \subseteq I^{t+1} R_Q$ and $s$ becomes a unit in $R_Q$, we deduce that $f = (s f)/s \in I^{t+1} R_Q \subseteq I^t R_Q$. Hence, $f \in I^t R_Q \cap R$. 

Thus, it suffices to show that under either hypothesis, $f \in I^t R_P \cap R$, which is equivalent to showing 
\begin{equation}\label{colon_condition}
    I^{t+1} R_P :_{R_P} P R_P \subseteq I^t R_P.
\end{equation}
First, assume that $\operatorname{ht}(I) \ge 1$ and $\operatorname{depth}(\operatorname{gr}_{I R_P}(R_P)) \ge 1$. Let $(A, \mathfrak{m}) = (R_P, P R_P)$ and $J = I R_P$. Then $J$ is an $\mathfrak{m}$-primary ideal of $A$. Suppose for the sake of contradiction that there exists an element $f \in J^{t+1} : \mathfrak{m}$ with $f \notin J^t$. Let $s$ be the largest integer such that $f \in J^s$. Then $s < t$ and $\bar{f} \in J^s / J^{s+1}$ is a non-zero element in $\operatorname{gr}_J(A)$. 

Let $M = \mathfrak{m}/J \oplus \bigoplus_{k \ge 1} J^k / J^{k+1}$ be the unique homogeneous maximal ideal of $\operatorname{gr}_J(A)$. We will show that $M \bar{f} = 0$. Indeed, for any $a \in \mathfrak{m}$, we have $(a + J)(f + J^{s+1}) = a f + J^{s+1}$. Since $a \in \mathfrak{m}$, we have $a f \in \mathfrak{m} f \subseteq J^{t+1} \subseteq J^{s+2}$ (since $t \ge s + 1$), so $\bar{a} \bar{f} = 0$. 

Now, for any $g \in J^k$ with $k \ge 1$, we have $J^k f = J^{k-1} J f \subseteq J^{k-1} \mathfrak{m} f \subseteq J^{k-1} J^{t+1} = J^{t+k}$. Therefore, $(g + J^{k+1})(f + J^{s+1}) = g f + J^{s+k+1} = 0 \in \operatorname{gr}_J(A)$ because $t+k \ge s+k+1$. This implies that $M \bar{f} = 0$, which contradicts the assumption that $\operatorname{depth}(\operatorname{gr}_J(A)) \ge 1$.

Now, assume that $R = k[x_1, \ldots, x_n]$ is a polynomial ring over a field $k$ and $I$ is a monomial ideal of $R$. Since the minimal primes of $I$ are generated by variables, localizing at $P$ is equivalent to setting all variables outside of $P$ to $1$. Condition~\eqref{colon_condition} is then equivalent to asking that $I^{t+1} : \mathfrak{m} \subseteq I^t$ holds in a polynomial ring, where $\mathfrak{m}$ is generated by the variables in $P$. Since $I^{t+1} : \mathfrak{m} = \bigcap_{i=1}^n (I^{t+1} : x_i)$, it suffices to show that $I^{t+1} : x_i \subseteq I^t$ for each variable $x_i \in \mathfrak{m}$. 

Indeed, let $m \in I^{t+1} : x_i$ be a monomial. Then $m x_i \in I^{t+1}$, so there exist generators $f_1, \ldots, f_{t+1} \in I$ and a monomial $h \in R$ such that $m x_i = f_1 \cdots f_{t+1} h$. Thus, $x_i$ divides either $h$ or one of the generators $f_j$. If $x_i \mid h$, then $m$ is divisible by $f_1 \cdots f_{t+1} \in I^{t+1} \subseteq I^t$. If $x_i \mid f_j$, then $f_j / x_i \in R$, so $m = (f_j / x_i) \prod_{l \neq j} f_l h \in I^t$. In either case, $m \in I^t$, completing the proof.
\end{proof}

\begin{theorem} 
Assume that $R$ is a standard graded Noetherian ring and $I$ is a proper ideal of $R$. Assume that either of the following conditions holds:
\begin{enumerate}
    \item $\operatorname{depth}(\operatorname{gr}_I(R)) \ge 1$;
    \item $R$ is a polynomial ring over a field and $I$ is a monomial ideal.
\end{enumerate}
Then $\v(I^r / I^q) = \v(R / I^q)$ for all $1 \le r < q$.
\end{theorem}

\begin{proof}
The conclusion follows from the proof of Theorem~\ref{theorem_filtered_I}.
\end{proof}
\section{The case of edge ideals of graphs}\label{sec_edge}

Let $G$ be a simple graph on the vertices $V(G) = [n] = \{1,\ldots,n\}$ and edge set $E(G)$. Let $R = k [x_1,\ldots,x_n]$ be a polynomial ring over a field $k$. The edge ideal of $G$ is 
$$I(G) = (x_ix_j \mid \{i,j\} \in E(G)\}.$$
In this section, we give a more detail description of $\v$-number of $M_t(I(G))$ in terms of combinatorial data of $G$. First, we recover the result of Muta \cite[Corollary 4.10]{m26} (also see \cite{hm26}).

\begin{corollary} 
Let $I = I(G)$ be the edge ideal of a simple graph $G$. Then 
$$\dim(M_t(I(G))) = \max \{ \alpha(G - N_G[C]) \mid C \text{ is an odd cycle of length } \le 2t-1 \}.$$    
\end{corollary}

\begin{proof}
By Lemma~\ref{lem_M_localization} and Theorem~\ref{thm_ass_M}, we deduce that $\dim(M_t(I(G))) = \max \{ \dim(R/Q) \mid (I R_Q)^t \neq (I R_Q)^{(t)} \}$ for prime ideals $Q$ corresponding to vertex covers of $G$. Let $F = V(G) \setminus C$. Then $I R_Q = I(G - N_G[F]) R_Q + (x_j \mid j \in N_G(F)) R_Q$. The latter condition is equivalent to requiring that $G - N_G[F]$ contains an odd cycle of length at most $2t-1$. The conclusion follows.
\end{proof}

\begin{theorem}\label{thm_v_2}
Let $G$ be a simple graph. By convention, we set $\v(I(G)) = 0$ when $G$ has no edges. Then 
$$\v(M_2(I(G))) = \min \{ 3 + \v(I(G - N_G[T])) \mid T \text{ is a triangle in } G \}.$$
\end{theorem}

\begin{proof} 
Let $I = I(G)$. By Theorem~\ref{thm_ass_M}, we have 
$$\v(M_2(I(G))) = \min \{ \v_P(I^2) \mid P \text{ is an embedded prime of } I^2 \}.$$
Let $f$ be a monomial such that $I^2 : f = P$ is an embedded prime of $I^2$. By the proof of Theorem~\ref{thm_ass_M}, we must have $f \in I^{(2)} \setminus I^2$. Therefore, there exists a triangle $T = \{1,2,3\}$ in $G$ such that $f = x_1 x_2 x_3 g$ for some monomial $g$. Since $f \notin I^2$, $\supp(g)$ is an independent set and $\operatorname{supp}(g) \cap N_G[T] = \emptyset$. 

Now, we have 
$$I^2 : f = I(G - N_G[T]) : g + (x_j \mid j \in N_G[T]).$$
Hence, $I^2 : f$ is generated by variables if and only if $I(G - N_G[T]) : g$ is generated by variables, which means $I(G - N_G[T]) : g$ is an associated prime of $I(G - N_G[T])$. Taking the minimum over all such monomials $g$ and triangles $T$, the conclusion follows.
\end{proof}

We will now compute the $\v$-number of $I(G)^2$ and classify the graphs for which $\v(I(G)^2) = \v(I(G)^{(2)})$. We first introduce some notation. Let $\mathbf{a} = (a_1, \ldots, a_n) \in \mathbb{N}^n$ be an exponent vector. For a subset $U \subseteq [n]$, we set $a(U) = \sum_{i \in U} a_i$. The support of $\mathbf{a}$ is defined as $\operatorname{supp}(\mathbf{a}) = \{ i \mid a_i \neq 0 \}$. We recall the following criterion from \cite{chjv26}.

\begin{lemma}
Let $I$ be a squarefree monomial ideal and let $C$ be a minimal vertex cover of $H(I)$, the facet hypergraph of $I$. Then $I^{(t)} : x^{\mathbf{a}} = P_C$ if and only if $a(C) = t - 1$ and $a(D) \ge t$ for every minimal vertex cover $D \neq C$ of $H(I)$.    
\end{lemma}

Now, assuming that $I = I(G)$ is the edge ideal of a simple graph $G$, we will solve the optimization problem to provide a refined combinatorial description of $\v_P(I^{(2)})$.

\begin{lemma}\label{lem_c_2} 
Let $G$ be a simple graph and let $C$ be a minimal vertex cover of $G$. Set $F = V(G) \setminus C$. Let $f = x_c g$, where $c \in C$ and $\operatorname{supp}(g) \subseteq F$. Denote by $\mathbf{a}$ the exponent vector of $f$. Then $I^{(2)} : f = P_C$ if and only if 
$$a(N(b) \cap F) \ge 1 \quad \text{for all } b \in C \setminus \{c\}, \quad \text{and} \quad a(N(c) \cap F) \ge 2.$$
\end{lemma}

\begin{proof}
\textbf{Necessity:} For each $b \in C$, set $D = (C \setminus \{b\}) \cup (N(b) \cap F)$. Then $D$ is a vertex cover of $G$ with $C \not\subseteq D$, which implies $a(D) \ge 2$. When $b \neq c$, this yields $a(N(b) \cap F) \ge 1$; when $b = c$, this yields $a(N(c) \cap F) \ge 2$.

\textbf{Sufficiency:} Let $D$ be a vertex cover of $G$ such that $C \not\subseteq D$. Since $F$ is an independent set, we have $D \cap F \supseteq N(b) \cap F$ for all $b \in C \setminus D$. If $c \notin D$, then $a(D) \ge a(D \cap F) \ge a(N(c) \cap F) \ge 2$. If $c \in D$, then choose $b \in C \setminus D$. Then $b \neq c$, so $a(D) \ge a_c + a(N(b) \cap F) \ge 1 + 1 = 2$. The conclusion follows.
\end{proof}

From this, we deduce a finer result for the local $\v$-number of $I^{(2)}$.

\begin{theorem} 
Let $G$ be a simple graph and let $C$ be a minimal vertex cover of $G$ with corresponding minimal prime $P_C$ of $I(G)$. Then 
$$\v_{P_C}(I(G)^2) = \v_{P_C}(I(G)^{(2)}).$$    
\end{theorem}

\begin{proof}
By \cite[Lemma 2.4]{chjv26}, it suffices to prove that $\v_{P_C}(I(G)^{(2)}) \le \v_{P_C}(I(G)^2)$. We will show that if $u = x^{\mathbf{a}}$ is a monomial such that $I(G)^{(2)} : u = P_C$, then $I(G)^2 : u = P_C$. 

By Lemma~\ref{lem_c_2}, we may write $u = x_c g$ with $\operatorname{supp}(g) \subseteq F = V(G) \setminus C$. Since $I^2 \subseteq I^{(2)}$, we have $I^2 : u \subseteq I^{(2)} : u = P_C$. Thus, it suffices to show that $x_b \in I^2 : u$ for all $b \in C$. 

If $b = c$, then $a(N(c) \cap F) \ge 2$ implies that $x_b u = x_c^2 g \in I^2$. If $b \neq c$, then by Lemma~\ref{lem_c_2}, we have $a(N(b) \cap F) \ge 1$, so we may pick $d \in N(b) \cap F$ such that $a_d \ge 1$. Since $a(N(c) \cap F) \ge 2$, we can pick $d' \in N(c) \cap F$ such that $x_d x_{d'} \mid g$. Therefore, $x_b u = (x_b x_d)(x_c x_{d'}) (g / x_d x_{d'}) \in I^2$. Hence, $x_b u \in I^2$ for all $b \in C$, completing the proof.
\end{proof}

To state a formula for the $\v$-number of $I(G)^2$, we introduce a few definitions.

\begin{definition} 
Let $G$ be a simple graph. A subset $S \subseteq V(G)$ is called a $\v$-witness of $G$ if both $S$ and $V(G) \setminus N_G[S]$ are independent sets.    
\end{definition}

For a $\v$-witness set $S$, we define 
$$p(S) = \begin{cases} 
1 & \text{if } \{N_G(s)\}_{s \in S} \text{ are pairwise disjoint,} \\
0 & \text{otherwise.}
\end{cases}$$
A $\v$-witness $S$ whose neighborhoods $\{N_G(s)\}_{s \in S}$ are pairwise disjoint is called a $2$-packing. 

We then obtain the following results.

\begin{theorem}\label{thm_equality_v_2} 
Let $G$ be a simple graph. Then:
\begin{enumerate}
    \item $\v(I(G)^{(2)}) = \min \{ |S| + p(S) \mid S \text{ is a } \v\text{-witness of } G \}+1$.
    \item $\v(I(G)^2) = \min \{ \v(I(G)^{(2)}), \v(M_2(I(G))) \}$.
\end{enumerate}
Hence, $\v(I(G)^2) < \v(I(G)^{(2)})$ if and only if there exist a triangle $T$ and a $\v$-witness $S'$ of $G - N_G[T]$ such that:
\begin{enumerate}
    \item $G$ has no $\v$-witness set of size at most $|S'| + 1$;
    \item Every $\v$-witness set of $G$ of size $|S'| + 2$ is a $2$-packing.
\end{enumerate}    
\end{theorem}

\begin{proof} 
Let $f$ be a monomial such that $I(G)^{(2)} : f = P_C$ is an associated prime of $I(G)^{(2)}$. By Lemma~\ref{lem_c_2}, we can write $f = x_c x^{\mathbf{a}}$ for some $c \in C$ and $\operatorname{supp}(\mathbf{a}) \subseteq V(G) \setminus C$. Set $S = \operatorname{supp}(\mathbf{a})$. Lemma~\ref{lem_c_2} implies that $N_G(S) = C$. If the neighborhoods $\{N_G(s)\}_{s \in S}$ are pairwise disjoint, then at least one variable $x_s$ in $x^{\mathbf{a}}$ must have an exponent of at least $2$, since $a(N_G(c) \cap S) \ge 2$. If there exist distinct $s_1, s_2 \in S$ sharing a common neighbor $c'$, then we can replace $c$ with $c'$ to find a monomial $g$ such that $I(G)^{(2)} : g = P_C$ with $\deg(g) \le \deg(f)$. The formula for $\v(I(G)^{(2)})$ follows.

The final equivalence follows from Theorem~\ref{thm_v_2} and the formula for $\v(I(G)^{(2)})$.
\end{proof}

To establish a formula for $\v(M_3(I(G)))$, we need to describe the colon ideal $I^3 : f$ for $f \in I^{(3)}$. We first establish some preliminary lemmas.

\begin{lemma}\label{lem_colon_K_4}
Let $G$ be a simple graph and let $K = \{i, j, k, l\}$ induce a complete subgraph $K_4$ of $G$. Assume that $f = x_i x_j x_k x_l g \notin I^3$ with $N_G(K) \cap \operatorname{supp}(g) = \emptyset$. Then $I^3 : f$ is not a prime ideal.    
\end{lemma}

\begin{proof} 
Since $f \notin I^3$, $\operatorname{supp}(g)$ is an independent set. Clearly, $I \subseteq I^3 : f$. Suppose for the sake of contradiction that $x_i \in I^3 : f$. Then $x_i f = x_i^2 x_j x_k x_l g \in I^3$. Thus, at least one of the edges in a product representation of $x_i f$ must meet $\operatorname{supp}(g)$, implying that $N_G(K) \cap \operatorname{supp}(g) \neq \emptyset$, which is a contradiction. By symmetry, $x_i, x_j, x_k, x_l \notin I^3 : f$. Therefore, $I^3 : f$ cannot be a prime ideal; otherwise, it would contain $I(G)$ without containing any of the vertices $\{i, j, k, l\}$ of the $K_4$, failing to correspond to a vertex cover of $G$.
\end{proof}

\begin{lemma}\label{lem_colon_c_5} 
Let $G$ be a simple graph and let $C$ be an induced $5$-cycle $C_5$ in $G$. Assume that $f = x_C g \notin I^3$, where $x_C = \prod_{v \in V(C)} x_v$. Then $N_G[C] \cap \operatorname{supp}(g) = \emptyset$ and 
$$I^3 : f = (x_j \mid j \in N_G[C]) + (I(G - N_G[C]) : g).$$
\end{lemma}

\begin{proof}
Since $f \notin I^3$, we have $N_G[C] \cap \operatorname{supp}(g) = \emptyset$. The right-hand side is clearly contained in the left-hand side. Conversely, if $x_l \in I^3 : f$, then $N_G(l) \cap \operatorname{supp}(f) \neq \emptyset$. The conclusion follows.
\end{proof}

\begin{lemma}\label{lem_colon_t_e}
Let $G$ be a simple graph and let $T$ be a triangle in $G$. Assume that $f = x_T g$ with $g \in I$ is such that $I^3 : f$ is a prime ideal. Let $W = \operatorname{supp}(g)$, $W_T = W \cap T$, and $W_N = (W \cap N_G[T]) \setminus T$. Then $|W_T \cup W_N| \le 2$ and one of the following must hold:
\begin{enumerate}
    \item $W_T \cup W_N = \emptyset$ and $I^3 : f = (x_j \mid j \in N_G[T]) + I(G - N_G[T])^2 : g$;
    \item $W_T \cup W_N = \{i\}$ with $i \in W_N$. Writing $f = x_T x_i u$, we have $I^3 : f = (x_j \mid j \in N_G[T]) + I(G - N_G[T]) : u$;
    \item $W_T \cup W_N = \{p,q\}$, where $\{p,q\}$ is an edge of $G$. Writing $f = x_T x_p x_q h$, we have $I^3 : f = (x_j \mid j \in N_G[T \cup \{p,q\}]) + I(G - N_G[T \cup \{p,q\}]) : h$.
\end{enumerate}
\end{lemma}

\begin{proof} 
If $|W_T \cup W_N| \ge 3$, say $i, j, k \in W_T \cup W_N$, then there exists at least one variable, say $x_i$, such that $g / x_i \in I$ since $g \in I$. Since $x_i x_T \in I^2$, this implies that $f = (x_i x_T) (g / x_i) \in I^3$, a contradiction. Thus, $|W_T \cup W_N| \le 2$.

Now assume that $W_T \cup W_N = \emptyset$. Then $\operatorname{supp}(g) \subseteq V(G) \setminus N_G[T]$. Clearly, $I^3 : f \supseteq (x_j \mid j \in N_G[T]) + (I(G - N_G[T])^2 : g)$. Conversely, let $h$ be a monomial such that $\operatorname{supp}(h) \subseteq V(G) \setminus N_G[T]$ and $h f \in I^3$. Then the three edges involved in a product representation of $h f$ must include at least two edges whose supports lie entirely in $V(G) \setminus N_G[T]$, as there are no edges between $T$ and $V(G) \setminus N_G[T]$.

Next, assume that $W_T \cup W_N = \{i\}$. Write $f = x_T x_i u$. Then $\operatorname{supp}(u)$ is an independent set since $x_i x_T \in I^2$ and $\operatorname{supp}(u) \subseteq V(G) \setminus N_G[T]$. Since $x_i u \in I$, we have $i \in W_N$. We claim that $I^3 : f = (x_j \mid j \in N_G[T]) + (I(G - N_G[T]) : u)$. The inclusion $(\supseteq)$ is clear. Now let $h$ be a monomial such that $h \in I^3 : f$ and $\operatorname{supp}(h) \subseteq V(G) \setminus N_G[T]$. Since $h f \in I^3$, there exist three edges whose product divides $h f$. Because $x_i$ appears with multiplicity $1$, it can be used in at most one edge. At most one edge can lie in $T$. Hence, the remaining edge must divide $h u$, so $h u \in I$. The conclusion follows.

Finally, assume that $W_T \cup W_N = \{p,q\}$. Write $f = x_T x_p x_q h$. Then $x_p h \notin I$ and $x_q h \notin I$, for otherwise $x_T x_q \in I^2$ would imply $f \in I^3$. Since $x_p x_q h \in I$, $\{p,q\}$ must be an edge of $G$. Clearly, $I^3 : f \supseteq (x_j \mid j \in N_G[T \cup \{p,q\}]) + (I(G - N_G[T \cup \{p,q\}]) : h)$. Now let $k$ be any minimal generator of $I^3 : f$ with $\operatorname{supp}(k) \cap N_G[T \cup \{p,q\}] = \emptyset$. Since $k$ is minimal, an element of $\operatorname{supp}(k)$ must be used in one of the edges in a product dividing $k f$, which can only connect to $\operatorname{supp}(h)$. The conclusion follows.
\end{proof}

\begin{theorem}\label{thm_v_3} 
Let $G$ be a simple graph. Set 
\begin{align*}
    A &= 3 + \min \{ \v(I(G - N_G[T])^2) \mid T \text{ is a triangle in } G \text{ and } V(G) \setminus N_G[T] \text{ contains an edge} \}, \\
    B &= 5 + \min \{ \v(I(G - N_G[T \cup \{s\}])) \mid T \text{ is a triangle in } G \text{ and } \operatorname{dist}_G(T,s) = 2 \}, \\
    C &= 5 + \min \{ \v(I(G - N_G[C_5])) \mid C_5 \text{ is an induced } 5\text{-cycle in } G \}, \\
  D&= 5+\min\{\v(I(G-N_G[T\cup\{p,q\}])) \mid  T\text{ is a triangle in }G \text{ and } \\
        & \qquad \qquad \qquad \qquad \qquad \qquad \qquad \qquad \qquad \qquad \qquad \{p,q\}\in E(G) \text{ with }p,q\in N_G[T] \}.
\end{align*}
Then 
$$\v(M_3(I(G))) = \min \{ A, B, C, D \}.$$
\end{theorem}

\begin{proof}
By Theorem~\ref{thm_v_M}, we have 
$$\v(M_3(I(G))) = \min \{ \deg(f) \mid f \in I^{(3)} \setminus I^3 \text{ and } I^3 : f \text{ is prime} \}.$$
By \cite[Theorem 2.10]{mnptv22}, $f$ takes one of the following forms: $f = x_T g$ with $g \in I$ and $T$ a triangle in $G$; $f = x_K g$ with $K$ an induced $K_4$ in $G$; or $f = x_C g$ with $C$ an induced $C_5$ in $G$. By Lemma~\ref{lem_colon_K_4}, if $f = x_K g$, then $N_G(K) \cap \operatorname{supp}(g) \neq \emptyset$, which can be rewritten in the first form. The conclusion then follows from Lemma~\ref{lem_colon_c_5} and Lemma~\ref{lem_colon_t_e}.
\end{proof}

\subsection*{Acknowledgement} The authors would like to thank Prof. Akiyoshi Sannai for his helpful suggestions, which led to corrections to the formulas stated in Theorems 4.7(1) and 4.11.


\bibliographystyle{plain}
\bibliography{reference}
\end{document}